\documentclass[a4paper,12pt]{article}
\usepackage[latin1]{inputenc}
\usepackage[T1]{fontenc}
\usepackage[english]{babel}

\usepackage{mathrsfs}
\usepackage{amscd}
\usepackage{amsfonts}
\usepackage{amsmath}
\usepackage{amssymb}
\usepackage{amstext}
\usepackage{amsthm}
\usepackage{amsbsy}

\usepackage{xspace}
\usepackage[all]{xy}
\usepackage{graphicx}
\usepackage{url}
\usepackage{latexsym}

\usepackage{graphicx} 

\usepackage{booktabs} 
\usepackage{array} 
\usepackage{paralist} 
\usepackage{verbatim} 
\usepackage{subfig} 

\usepackage{fancyhdr} 
\usepackage{sectsty}
\allsectionsfont{\sffamily\mdseries\upshape} 

\usepackage[nottoc,notlof,notlot]{tocbibind} 
\usepackage[titles,subfigure]{tocloft} 

\usepackage[textwidth=100pt,textsize=footnotesize,bordercolor=white,color=blue!30]{todonotes}
\usepackage{hyperref} 

\makeatletter
\newcommand*{\rom}[1]{\expandafter\@slowromancap\romannumeral #1@}
\makeatother

\theoremstyle{definition}

\newtheorem{fact}{fact}

\newtheorem{thm}[fact]{Theorem}
\newtheorem{lemma}[fact]{Lemma}
\newtheorem{prop}[fact]{Proposition}
\newtheorem{corollary}[fact]{Corollary}
\newtheorem{defini}[fact]{Definition}

\title{Structures Associated with Real Closed Fields and the Axiom of Choice}
\author{Merlin Carl}
\date{}

\begin{document}

\maketitle


\begin{abstract}
An integer part $I$ of a real closed field $K$ is a discretely ordered subring of $K$ with minimal positive element $1$
such that, for every $x\in K$, there is $i\in I$ with $i\leq x<i+1$.
Mourgues and Ressayre showed in \cite{MR} that every real closed field has an integer part. Their construction implicitly uses the axiom of choice.
 We show that $AC$ is actually necessary to obtain the result by constructing a transitive model of $ZF$ which contains a real
closed field without an integer part. Then we analyze some cases where the axiom of choice is not necessary for obtaining an integer part.
On the way, we demonstrate that a class of questions containing the question whether the axiom of choice is necessary for the proof of a certain $ZFC$-theorem
is algorithmically undecidable. We further apply the methods to show that it is independent of $ZF$ whether every real closed field has a value group section and a residue field section.
 This also sheds some light on the possibility to effectivize constructions of integer parts and value group sections which was considered e.g. in \cite{DKKL}
and \cite{KL}.
\end{abstract}

\section{Introduction}



A real closed field ($RCF$) $K$ is a field in which $-1$ is not a sum of squares and every polynomial of odd degree has a root. Equivalently, it is
elementary equivalent to the field of real numbers in the language of rings.
We assume familiarity with the basic notions and theorems connected with $RCF$s and refer the reader to \cite{CK} otherwise.
A field $K$ is formally real if $-1$ is not a sum of squares in $K$. $K$ is orderable if there is a linear ordering $\leq$ of $K$
that respects the addition and multiplication of $K$. If $K$ is formally real, then there is a real closed algebraic field extension $K^{\prime}$ of $K$,
called real closure of $K$, which is unique when $K$ is orderable, in which case we will denote it by $K^{rc}$. The existence of real closures for formally real fields depends on the axiom of choice,
while the existence of real closures for ordered fields is known to follow from $ZF$ alone.
If $K$ is a real closed field, $X\subseteq K$ and $K_{X}$ is the smallest subfield of $K$ containing $X$ as a subset,
then $K_{X}^{rc}$ is also a subfield of $K$.

\begin{defini}
Let $K$ be an $RCF$. Then $I\subseteq K$ is an integer part of $K$ iff $I$ is a discretely ordered subring of $K$ such that $1$ is the minimal positive element
of $I$ and, for every $x\in K$, there is $i\in I$ with $i\leq x<i+1$.
\end{defini}

The idea here is that $I$ is in a relation to $K$ similar to the relation of $\mathbb{N}$ and $\mathbb{R}$.
Integer parts of real closed fields are especially interesting as they are known to coincide with models of a certain natural fragment of Peano Arithmetic,
namely Open Induction (see \cite{S}).\\

In \cite{MR}, Mourgues and Ressayre showed that every real closed field has an integer part. Their construction uses the axiom of choice (in the form of Zorn's lemma) implicitly at
least in the proof of the crucial Corollary $4.2$. There has recently been some interest in the complexity of such a construction, see e.g. \cite{DKL} or \cite{DKKL}.
For this purpose, a well-ordering of the real closed field is assumed to be given. This motivates us to ask whether this ingredient is actually necessary, i.e.
whether there is a way to `construct' an integer part from the real closed field alone. Furthermore, there are other structures associated with real closed fields
that are used in the the Mourgues-Ressayre construction, namely value group sections and residue field sections. The usual arguments for their existence
uses Zorn's Lemma, and the question whether their construction can be effectivized has been studied in the countable case in \cite{KL}. Here we ask the same question:
Is the axiom of choice necessary to prove the existence of value group sections and residue field sections for real closed fields?\\

As we show in section $3$ as a side note, such questions are in general difficult to answer: We observe a general theorem that in particular implies that the question
whether a given $ZFC$-theorem can be proved in $ZF$ alone is not algorithmically decidable.\\

Concerning the theorems mentioned above, the axiom of choice turns out to be indeed necessary: In section $4$, we construct
transitive models of Zermelo-Fraenkel
set theory without the axiom of choice ($ZF$) containing a real closed field $K$, but no integer part of $K$. Letting $\phi_{IP}$ denote
the statement that every real closed field has an integer part, this shows that $\phi_{IP}$ is independent of $ZF$. In section $5$, we give some extra conditions on real closed fields
under which the axiom of choice is not necessary for obtaining an integer part. Furthermore, we construct models of $ZF$ containing real closed fields without
value group sections and residue field sections. In section $4.2$, we lay bare the combinatorial core of these constructions, proving two rather general statements about
subsets of models of $o$-minimal theories that provably exist under $ZF$. In section $6$, we obtain as a consequence that the existence of integer parts,
value group sections and residue field sections is indeed highly non-constructive: To be precise, we show that none of them is realizable via primitive recursive set functions.
In section $7$, we give an upper estimate on the amount of choice necessary to show that a real closed field possesses these substructures, showing in $ZF$ that
well-orderable real closed fields have value group sections. Section $8$ then briefly discusses some open questions and topics for 
further work.

\section{Preliminaries}

We briefly summarize some notions from set theory that are necessary to understand the central tool for
constructing a choice-free universe with an $RCF$ without an integer part, namely Lemma \ref{subsetremoval} below, originally
used by Hodges to show the dependence of several algebraic constructions on the axiom of choice (\cite{H1}).
$ZF$ is Zermelo-Fraenkel set theory, $ZFC$ is $ZF$ together with the axiom of choice, see e.g. Chapter $1$ of \cite{J1}.
A model of $ZF(C)$ is transitive iff $x\in M$ whenever $y\in M$ and $x\in y$. The rough idea is to start with a given countable
transitive model $M$ of $ZFC$ (the existence of such models is known to be consistent with $ZFC$ unless $ZFC$ itself is not)
containing a certain algebraic object $A$ and then build a model $N(A)$ of $ZF$ that contains an isomorphic copy $A^{\prime}$ of $A$ (i.e. in the real world $V$,
there is an isomorphism between $A$ and $A^{\prime}$ - for our purposes, we can identify $A$ and $A^{\prime}$), but only those subsets of $A$
that are respected by all automorphisms that fix some finite subset of $A$. It is important to note that `all automorphisms' here is to be understood
with respect to the set-theoretical universe $V$, in particular not relativized to the model $N(A)$.\\
Most of our notation is standard. If $f$ is a map and $X$ is a subset of its domain, then we denote by $f[X]$ the image of $X$ under $f$;
if $\vec{a}=(a_1,...,a_n)$ and $\vec{b}=(b_1,...,b_n)$ are finite subsets of the domain and the range of $f$, respectively, then
we write $f(\vec{a})=\vec{b}$ to mean that $f(a_{i})=b_{i}$ for all $i\in\{1,...,n\}$. When we write ($ZF$) before a theorem statement,
we mean that the statement is provable in $ZF$ alone, i.e. without the axiom of choice. When $\mathfrak{A}$ is a structure and $\vec{x}\subseteq\mathfrak{A}$ is
a finite sequence of elements of $\mathfrak{A}$,
then $\text{tp}_{\mathfrak{M}}(\vec{x})$ denotes the type of $\vec{x}$ in $\mathfrak{A}$, where the subscript $\mathfrak{M}$ is dropped when the structure is clear
from the context. For a set $X$ and a transitive $S\subseteq M\models ZF$, $\mathfrak{P}^{M}(X)$ denotes the power set of $X$ in $M$,
i.e. the set of subsets of $X$ contained in $M$.

\begin{defini}{\label{Msymmetric}}
 Let $M\models ZF$, and let $R\in M$ be a ring. Then $R$ is $M$-symmetric iff, for every $X\in\mathfrak{P}^{M}(R)$,
there is a finite $s_{X}\subseteq R$ such that $\pi[X]=X$ for every automorphism $\pi$ of $R$ which fixes $s_{X}$ pointwise, i.e.
$\pi_{X}(a)=a$ for every $a\in s$. In this case, $s$ is called a support for $X$.
\end{defini}

\textbf{Remark}: Note that, if $R$ is an $RCF$, then $\pi$ must preserve the canonical ordering of $R$ so that for finite $s$,
$\pi$ fixes $s$ pointwise iff $\pi(s)=s$. 

\begin{lemma}{\label{subsetremoval}}
 Let $L$ be a countable first-order language and $\mathfrak{A}$ be a countable $L$-structure. Then there is a transitive model $N(\mathfrak{A})$ of $ZF$
which contains an $N(\mathfrak{A})$-symmetric isomorphic copy of $\mathfrak{A}$.
\end{lemma}
\begin{proof}
 This is Lemma $3$ of \cite{H1} and also Lemma $3.7$ of \cite{H2}, where it is proved.
\end{proof}

That $N(\mathfrak{A})$ contains an isomorphic copy $\mathfrak{A}^{\prime}$ of $\mathfrak{A}$ is in general not enough to show an independence from $ZF$ - one
must also know that the relevant properties of $\mathfrak{A}$ still hold for $\mathfrak{A}^{\prime}$ in $N(\mathfrak{A})$. Some properties, as e.g. 
countability, will in general not be preserved. However, every notion relevant for our purposes is absolute between transitive models of $ZF$ (see section $4$ for the definitions
of a value group section and a residue field section):

\begin{lemma}{\label{absoluteness}}
Let $K$ and $R$ be sets. The following statements are absolute between all transitive models of $ZF$ containing $K$ (or $K$ and $R$, where relevant):
\begin{itemize}
\item (a) $K$ is a real closed field
\item (b) $R$ is a subring of $K$
\item (c) $R$ is a dense subset of $K$
\item (d) $R$ is an integer part of $K$
\item (e) $R$ is a value group section of $K$
\item (f) $R$ is a residue field section of $K$
\end{itemize}
\end{lemma}
\begin{proof}
Recall that a formula in the language of set theory is $\Delta_{0}$ if all of its quantifiers are bounded.
By Lemma $12.9$ of \cite{J1}, $\Delta_{0}$-formulas are absolute between transitive classes. It is
easy to see that (b)-(f) are expressable with quantifiers restricted to $K$ and $R$, so that
(b)-(f) are in fact absolute between arbitrary transitive classes containing $K$ and $R$.\\
It remains to see that (a) is $\Delta_{0}$-expressable in a transitive model of $ZF$. 
This is obvious for the axioms of ordered fields. We need to say that $-1$ is not a sum of squares in $K$
and that every polynomial of odd degree with coefficients in $K$ has a root in $K$. Denoting by $^{<\omega}X$
the set of finite sequences of elements of a set $X$, this is easily expressable with quantifiers
bounded by $^{<\omega}(\omega\cup K)$. But $^{<\omega}(\omega\cup K)$ exists and is absolute for
every transitive model of $ZF$ containing $K$.
\end{proof}

\textbf{Remark}: In order for some of our statements to make sense, we note that the value group and the residue field of a real closed field $K$, being definable over $K$, exist in plain $ZF$.
Hence e.g. a model of $ZF$ with an $RCF$ $K$ without a value group section will contain $K$ as well as its value group $\mathcal{O}$, but no embedding from $\mathcal{O}$ into $K$.

\section{A Remark on the Decidability of the Necessity of Axioms}

We take the opportunity to remark that there is no general procedure to decide whether or not the axiom of choice is needed in the proof of a certain $ZFC$-theorem.
In fact, there are various places in mathematical logic where one is concerned with the necessity of certain axioms for the proof of a theorem; in set theory, typical questions are about
the necessity of the axiom of choice and large cardinal assumptions. In arithmetic, one is interested in the minimal degree of induction necessary for the proof of some statement.
Generally, it is hard to determine the answer. This suggests that these problems may be undecidable. We show that this is indeed the case for all the cases mentioned and in fact many more.
The proof is quite easy; to the best of our knowledge, however, this has so far not been noted or written down. 

\begin{thm}{\label{decidingweakertheories}}
Let $T$ be a first-order theory and $\phi$ a statement such that $T+\neg\phi$ is undecidable. Then there is no effective procedure to decide whether, given a proof in $T+\phi$ of some statement $\psi$,
the statement $\psi$ is provable in $T$; that is, there is no program $P$ such that $P$, when run on the G\"odel code of a $T+\phi$-provable formula $\psi$, stops with output $1$ iff $\psi$
is provable in $T$, stops with output $0$ iff $\psi$ is not provable in $T$ and does not stop when run on the G\"odel code of a formula not provable in $T+\phi$.
\end{thm}
\begin{proof}
Let $T$ and $\phi$ be as in the assumptions of the theorem. Assume for a contradiction that $P$ is a program as described and let $\psi$ be an arbitrary formula.
Clearly, $T+\phi$ proves $\psi$ iff $T$ proves $\phi\rightarrow\psi$. Now 
$\phi\vee\psi$ is clearly provable in $T+\phi$ for every
$\psi$. Hence, we can use $P$ to decide whether for a given formula $\psi$, the formula $\phi\vee\psi$ is provable in $T$. But $\phi\vee\psi$ is provable in $T$ iff
$\psi$ is provable in $T+\neg\phi$. Thus $P$ can be used to decide $T+\neg\phi$, which contradicts the assumption that $T+\neg\phi$ is undecidable.
\end{proof}

\textbf{Remark}: This result can also be obtained as a consequence to a theorem of Ehrenfeucht and Mycielski which states that, if 
$T+\neg\alpha$ is undecidable, then there is no recursive function $f$ such that $W_{T}(\phi)\leq f(W_{T+\alpha}(\phi))$ holds for all theorems $\phi$ of $T$; here, $W_{S}$
is a measure for the complexity of the shortest proof of formula in the theory $S$. For a precise notion of a complexity measure and a proof, see \cite{EM}.

We note some particularly interesting special cases. (We assume that $PA$ and $ZFC$ are consistent.)

\begin{corollary}
\begin{enumerate}
 \item There is no effective procedure to decide whether a $ZFC$-theorem is provable in $ZF$ alone.
 \item Assuming the consistency of some large cardinal hypothesis $H$, there is no effective procedure to decide whether a $ZFC+H$-theorem is provable in $ZFC$ alone
 \item There is no effective procedure that maps PA-theorems $\phi$ to the degree of induction necessary for their proof, i.e. the smallest $n$ such that $I\Sigma_{n}\vdash\phi$
\end{enumerate}
\end{corollary}
\begin{proof}
For (1) and (2), this follows from the observation that every consistent recursive extension of $ZF$ is undecidable.\\
Concerning (3), assume for a contradiction that $P$ is a program that, given a (code for) a $PA$-theorem $\psi$, outputs the smallest $n$ such that $\psi$ is provable in $I\Sigma_{n}$ (and does not halt when
the input is not provable in $PA$). We use the definability of bounded truth predicates in arithmetic to write the induction axioms for $\Sigma_n$ formulas as a single formula $\phi_{n}$ for each $n\in\omega$.
Now, every recursive consistent extension of $I\Sigma_{1}$ (i.e. $PA^{-}+\phi_{1}$) is undecidable. Since $\phi_{2}$ is not implied by $I\Sigma_{1}$, $PA^{-}+\phi_{1}+\neg\phi_{2}$
is consistent and hence undecidable. But, as $I\Sigma_{2}$-theorems are $PA$-theorems, $P$ would allow us to decide whether some theorem $\psi$ of $I\Sigma_{2}$ is provable in $I\Sigma_{1}$, a contradiction to Theorem \ref{decidingweakertheories}.
\end{proof}

\section{Real Closed Fields without Integer Parts}

We now construct a transitive $M\models ZF$ such that, for some $K\in M$, $M\models$`$K$ is a real closed field'$\wedge$`$K$ has no integer part'. Our method is that used by 
Hodges in \cite{H1} and \cite{H2} to construct choicefree counterexamples to some algebraic theorems.

\begin{defini}
If $K$ is a real closed field, then $X\subseteq K$ is bounded in $K$ iff there is $y\in K$ such that $y>x$ for every $x\in X$.
 An $RCF$ $K$ is unbounded iff for every finite $\vec{a}\subseteq K$, the real closure $\text{RC}(\vec{a})$ of $\vec{a}$ in $K$ is bounded in $K$.
$K$ is $\omega$-homogenous iff, for all finite $\vec{a},\vec{b}\subseteq K$, 
$\text{tp}(\vec{a})=\text{tp}(\vec{b})$ implies that for every $c\in K$, there is $d\in K$
such that $\text{tp}(\vec{a},c)=\text{tp}(\vec{b},d)$.
\end{defini}

The point behind the following lemma is that, if $K$ is an $RCF$, $\vec{a}\subseteq K$ finite and $x\in K$, then $\text{tp}(x,\vec{a})$
only depends on the place of $x$ in the ordering of $\text{RC}(\vec{a})$. For a direct proof of this, see Lemma $5.4.3$ of \cite{CK}.

\begin{lemma}{\label{closetypes}}
Let $K$ be an $RCF$, $a\subseteq K$ finite such that $\text{RC}(a)$ is bounded in $K$, $x>\text{RC}(a)$. 
Then there is $\varepsilon>0$ in $K$ such that $\text{tp}(a,x)=\text{tp}(a,y)$ for all $y\in(x-\varepsilon,x+\varepsilon)$.\\
In particular, if $I$ is an integer part of $K$, then there are $i\in I$, $r\in K\setminus I$ such that $\text{tp}(a,i)=\text{tp}(a,r)$.
\end{lemma}
\begin{proof}
By the claim in the proof of Theorem $3.4$ of \cite{DKS}, if $r\notin\text{RC}(\vec{a})$, then $\text{tp}(r,\vec{a})$ is realised
by all elements of an interval $C$ containing $r$.\\
The other statement follows easily as $\text{RC}(a)$ is bounded in $K$, while $I$ is unbounded
 and so $I\setminus\text{RC}(a)\neq\emptyset$. Hence we can pick some $i\in I$ bigger than all elements of $\text{RC}(\vec{a})$, pick $\varepsilon<\frac{1}{2}$ as in the first statement and take $r\in(i-\varepsilon,i+\varepsilon)\setminus K$.
Then $i$ and $r$ are as desired.
\end{proof}

\begin{lemma}{\label{homauto}}
Let $M$ be homogenous, $n\in \mathbb{N}$, $\vec{a},\vec{b}\in M^{n}$, $\text{tp}(\vec{a})=\text{tp}(\vec{b})$.
Then there is an automorphism $\pi$ of $M$ such that $\pi(\vec{a})=\vec{b}$. In particular, this holds when $M$ is countable and $\omega$-homogenous.
\end{lemma}
\begin{proof}
See Proposition $4.2.13$ of \cite{Ma}.
\end{proof}

\begin{lemma}{\label{constructionofautomorphisms1}}
 Let $K$ be a countable, unbounded, $\omega$-homogenous real closed field, let $\vec{a}\subseteq K$ be finite, and let $I$ be an integer part of $K$.
Then there is an automorphism $\pi$ of $K$ such that $\pi$ fixes $\vec{a}$ pointwise and there are $x\in I$, $y\in K\setminus I$
such that $\pi(x)=y$. In particular, $\vec{a}$ is not a support for $I$.
\end{lemma}
\begin{proof}
As $K$ is unbounded, $\text{RC}(\vec{a})$ is bounded in $K$. Hence, by Lemma \ref{closetypes}, there
are $i\in I$, $r\in K\setminus I$ such that $\text{tp}(\vec{a},i)=\text{tp}(\vec{a},r)$.
By Lemma \ref{homauto}, there is an automorphism $\pi$ of $K$ such that $\pi$ is the identity on $\vec{a}$
and $\pi(i)=r$.\\
Hence $\pi$ fixes $\vec{a}$, but $I\neq\{\pi(j)|j\in I\}$. Thus $\vec{a}$ is not a support for $I$.
\end{proof}

\begin{corollary}{\label{nosupport}}
If $K$ is a countable, unbounded, $\omega$-homogenous real closed field, then no integer part $I$ of $K$ has a support.
\end{corollary}
\begin{proof}
Immediate, as Lemma \ref{constructionofautomorphisms1} is true for all integer parts $I$ of $K$ and all finite $\vec{a}\subseteq K$.
\end{proof}

\begin{thm}{\label{countableandhomogenous}}
 Let $K$ be a countable, unbounded, $\omega$-homogenous $\text{RC}F$. Then there exists $M\models ZF$ containing
an isomorphic copy of $K$ such that $K$ has no integer part in $M$.
\end{thm}
\begin{proof}
 By Lemma \ref{constructionofautomorphisms1}, no integer part of $K$ can have a support. Hence we can apply Lemma \ref{subsetremoval} to
get a model $M$ of $ZF$ containing an isomorphic copy of $K$. Suppose that $M$ contains an integer part $I$ for $K$.
Then $I$ is in particular a subset of $K$ contained in $M$ and hence has a support $\vec{a}$. But this contradicts Corollary \ref{nosupport} (mind
our remarks on the meaning of `every automorphism' preceeding Definition \ref{Msymmetric}).
\end{proof}

\begin{corollary}{\label{IndependenceOfIPs}}
There are transitive models of $ZF$ which contain an RCF without an integer part. Consequently, $\phi_{IP}$ is independent of $ZF$.
\end{corollary}
\begin{proof}
Let $I\models PA$ be countable and nonstandard, and let $K$ be the real closure of its fraction field. By Proposition $3.3$ of \cite{DKS}, $K$ is unbounded.
Certainly, $K$ is countable. By Theorem $5.1$ of \cite{DKS}, $K$ is recursively saturated. By a Theorem of Barwise and Schlipf (see \cite{BS}),
countable recursively saturated structures are resplendent, and by Theorem $2.4$ (ii) of the same paper, resplendent structures are $\omega$-homogenous.
Hence, by Theorem \ref{countableandhomogenous},
there is a transitive $M\models ZF$ such that $M$ contains an isomorphic copy $K^{\prime}$ of $K$ without an integer part in $M$.
By Lemma \ref{absoluteness}, $K^{\prime}$ is an $\text{RC}F$ in $M$. Assume for a contradiction that $M\models$`$K^{\prime}$ has an integer part', and let
$I^{\prime}\in M$ such that $M$ thinks that $I^{\prime}$ is an integer part of $K^{\prime}$. By Lemma \ref{absoluteness} again, $I^{\prime}$
is then an integer part of $K^{\prime}$ in the real world contained in $M$, a contradiction. Hence $M$ believes that
$K^{\prime}$ is a real closed field without an integer part. Thus $\phi_{IP}$ is not provable in $ZF$.\\
On the other hand, $ZFC$ is shown to imply $\phi_{IP}$ in \cite{MR}. As $ZFC$ is consistent relative to $ZF$ (see e.g. Theorem $3.5$ of \cite{J}),
$\phi_{IP}$ is consistent with $ZF$. Thus $\phi_{IP}$ is independent of $ZF$.
\end{proof}

\subsection{Valuation-theoretical consequences}

As a byproduct of the considerations above, we get two consequences for the valuation theory of real closed fields.

\begin{defini}
 Let $G$ be a totally ordered abelian group. Then $x,y\in G$ are archimedean equivalent iff there is $n\in\mathbb{N}$ such that $x<ny$ and $y<nx$.
If $K$ is a real closed field, then $x,y\in K$ are called archimedean equivalent - written $x\sim y$ iff they are archimedean equivalent as elements of the totally ordered abelian group $(K,+)$.
\end{defini}

\begin{defini}
Let $K$ be a real closed field with value group $\theta(K^{\times})$. A value group section of $K$ is the image of a group embedding $f:\theta(K^{\times})\rightarrow K^{>0}$ from 
the value group of $K$ to its multiplicative group of positive elements that intersects each $\sim$-equivalence class in exactly one element.
\end{defini}

It was proved in \cite{Ka} (Theorem $8$) that every real closed field has a value group section with respect to the standard valuation.
Morever, it was shown in \cite{KL} that the construction of value group sections is $\Delta_{2}^{0}$ in $K$
for countable real closed fields $K$ and that this bound is strict. The proof in \cite{Ka} uses Zorn's Lemma. We show that this is actually a necessary
ingredient.

\begin{prop}{\label{valgroupdiscunb}}
 Let $K$ be a non-archimedean real closed field, and let $G\subseteq K$ be a value group section of $K$ with respect to the natural valuation with corresponding
embedding $t:\theta(K^{\times})\rightarrow K$. 
Then for each $x\in K$, there is $y\in G$ with $y>x$ and for each $y\in G$, there is $\varepsilon\in K$ such that $(y-\varepsilon,y+\varepsilon)\cap G=\{y\}$.
\end{prop}
\begin{proof}
 Let $x\in K$ be arbitrary. Assume without loss of generality that $x$ is infinite. Then $x^{2}$ is greater than all elements of $\mathbb{N}x$, hence the image of
$\theta(x^{2})$ under $t$ is greater than $x$.\\
Now let $y\in G$. Let $\varepsilon=\frac{y}{2}$. Then all elements of $(y-\varepsilon,y+\varepsilon)$ are archimedean equivalent and hence have the same image $v$ under $\theta$.
As $t$ is injective, $y$ is the only pre-image of $v$ under $t$, so $\text{im}(t)\cap(y-\varepsilon,y+\varepsilon)$ is as desired.
\end{proof}

\begin{corollary}
 Let $M$, $K$ and $K^{\prime}$ be as in the proof of Corollary \ref{IndependenceOfIPs}. Then $M\models `K^{\prime}$ is an RCF and $K^{\prime}$ has no value group section with respect to the standard valuation'.
\end{corollary}
\begin{proof}
By Lemma \ref{subsetremoval}, it suffices to show that no value group section of $K$ can have a support. To see this, let $G$ be a value group section of $K$ and let $\vec{a}\subseteq M$ be finite. As $K$ is unbounded by assumption and
by Proposition \ref{valgroupdiscunb}, there is $g\in G$ greater than all elements of RC$(\vec{a})$. By Lemma \ref{closetypes}, there is $\delta\in K$ such that tp$(\vec{a},x)$ is the same for all
$x\in(g-\delta,g+\delta)$. By Proposition \ref{valgroupdiscunb}, there is $\varepsilon\in K$ such that $(y-\varepsilon,y+\varepsilon)\cap G=\{y\}$.  Let $\varepsilon^{\prime}=\text{min}(\varepsilon,\delta)$ and pick
$z\in (g-\varepsilon^{\prime},g+\varepsilon^{\prime})\setminus\{g\}$ be arbitrary.
Applying Lemma \ref{homauto} to $(\vec{a},g)$ and $(\vec{a},z)$ and as $\varepsilon^{\prime}<\delta$, there is an automorphism $\pi$ of $K$ that fixes $\vec{a}$ and sends $g$ to $z$. As $\varepsilon^{\prime}<\varepsilon$,
it follows that $z$ is not an element of $G$. Hence $\pi$ witnesses that $\vec{a}$ is not a support for $G$. As $G$ and $\vec{a}$ were arbitrary, no value group section of $K$ has a support. Hence $M$ contains no value group section
for $K$. By Lemma \ref{absoluteness}, the claim follows.
\end{proof}

\begin{corollary}{\label{vgindep}}
 It is independent of $ZF$ whether each real closed field has a value group section with respect to its natural valuation.
\end{corollary}

We now turn to another structure commonly associated with a real closed field, namely residue field sections. Roughly, a residue field section of a real closed field contains the real numbers in that field; the residue field section forms
a maximal archimedean subfield and is hence isomorphic to a subfield of the reals. The relevant facts about residue field sections can be found e.g. in \cite{KL}.
Here is the formal definition:

\begin{defini}
 Let $K$ be an $RCF$, let $F$ be the set of its finite elements (i.e. $\{x\in K:\exists{n\in\mathbb{Z}}|x|<n\cdot 1\}$, usually called the `valuation ring' of $K$) and let $\mu$ be the set of infinitesimal elements of $K$
(i.e. $\{x\in K:\forall{x\in\mathbb{N}}|x|<\frac{1}{n}\}$). It is easy to see that $\mu$ is a maximal ideal of the ring $F$ and hence that the quotient $R:=F/\mu$ is a field, called the residue field of $K$. 
For $x,y\in K$, let us write $x\sim_{\mu}y$ if and only if $x-y\in\mu$. A residue field section
is the image of an embedding $\pi:R\rightarrow F$ that intersects each $\sim_{\mu}$-equivalence class of an element of $F$ in exactly one element.
\end{defini}

It is easy to prove using Zorn's Lemma that every $RCF$ has a residue field section (see e.g. Theorem $8$ of \cite{Ka}). We will now see that Zorn's Lemma is actually necessary.

\begin{lemma}{\label{inftranscdegAC}}
 Let $K$ be a countable, unbounded, $\omega$-homogenous real closed field such that the residue field $R$ of $K$ has infinite transcendence degree over $\mathbb{Q}$. Then
no residue field section of $K$ has a support.
\end{lemma}
\begin{proof}
 Let $S$ be a residue field section of $K$, and suppose that $\vec{a}\subseteq K$ is a support for $S$. As $K$ is unbounded, it contains an infinitesimal element $\alpha$; by
definition of a residue field section, if $r\in R$, then $(r-\alpha,r+\alpha)\cap S=\{r\}$, so $R$ is discrete in $K$. As $R$ has infinite transcendence degree over $\mathbb{Q}$ and $\vec{a}$ is finite,
$S$ is not a subset of $\text{RC}(\mathbb{Q}(\vec{a}))$; let $x\in S\setminus \text{RC}(\mathbb{Q}(\vec{a}))$. As in the proof of Lemma \ref{closetypes}, there is $\varepsilon>0$ such that for all elements $z$ of 
$I:=(r-\varepsilon,r+\varepsilon)$, $\text{tp}(a,r)=\text{tp}(a,z)$. We may assume without loss of generality that $\varepsilon\in\mu$ so that
$I\cap R=\{r\}$. Let $z\in I\setminus\{r\}$. Then there is, by $\omega$-homogenity, an automorphism $\pi$ of $K$ that fixes $a$ and sends $r$ to $z$. As $z\notin R$ by the choice of $z$, $a$ is not a support for $R$.
\end{proof}

\begin{thm}{\label{rfswithoutAC}}
 There is a transitive model $M$ of $ZF$ containing a real closed field $K^{\prime}$ without a residue field section. Consequently, $ZF$ does not prove that every real closed fields has a residue field section.
\end{thm}
\begin{proof}
By Lemma \ref{inftranscdegAC} and Lemma \ref{subsetremoval}, it suffices to construct a countable, unbounded, $\omega$-homogenous $RCF$ whose residue field has infinite transcendence degree over $\mathbb{Q}$. This can be achieved with an elementary chain
argument: Let $K_{0}$ be a subfield of $\mathbb{R}$ with infinite transcendence degree over $\mathbb{Q}$. By compactness, let $U_{0}$ be a countable elementary extension of $K_{0}$ containing an element greater than every element of $K_{0}$.
By Proposition $4.3.6$ of \cite{Ma}, let $H_{0}$ be a countable $\omega$-homogenous elementary extension of $U_{0}$. If $H_{0}$ is unbounded, then it is as desired. otherwise we set $K_{1}:=H_{0}$ and iterate the construction
to obtain a sequence $(K_{i}|i\in\mathbb{N})$ of $\omega$-homogenous real closed fields whose residue field has infinite transcendence degree over $\mathbb{Q}$, where $K_{j}$ contains an element greater than
every element of $K_{i}$ for $j>i$. Then $K:=\bigcup_{i\in\mathbb{N}}$ will be a real closed field (see e.g. Proposition $2.3.11$ of \cite{Ma}), it will be $\omega$-homogenous as a countable union of an
elementary chain of countable $\omega$-homogenous models, it will be countable as a countable union of countable sets and it will have infinite transcendence degree over $\mathbb{Q}$ as this already holds for the subset $K_{0}$.
Hence $K$ is as desired.
\end{proof}


\subsection{Generalizations}

\begin{defini}
 Let $(M,<)$ be an ordered structure. $X\subseteq M$ is unbounded iff for every $x\in M$, there is $y\in X$ such that $y>x$. $X\subseteq M$ is discrete iff, for every $x\in X$, there are $a,b\in M$ such that $a<x<b$ and $(a,b)\cap X=\{x\}$.
\end{defini}

\begin{defini}
Let $\mathfrak{L}$ be a first-order language and let $\mathfrak{M}$ be an $\mathfrak{L}$-structure. $x\in\mathfrak{M}$ is definable in $\mathfrak{M}$ iff there is an $\mathfrak{L}$-formula $\phi(v)$ in
one free variable $v$ such that $x$ is the only element $v$ of $\mathfrak{M}$ such that $\mathfrak{M}\models\phi(v)$. If $A\subseteq\mathfrak{M}$, then $x\in\mathfrak{M}$ is definable from $A$ iff
there exist an $\mathfrak{L}$-formula $\phi(v,\vec{w})$ and a finite sequence $\vec{a}\subseteq A$ such that $x$ is the unique element $v$ of $\mathfrak{M}$ with $\mathfrak{M}\models\phi(v,\vec{a})$.
The set of all elements definable from $A$ in $\mathfrak{M}$ is called the definable closure of $A$ in $\mathfrak{M}$, denoted $\text{dcl}_{\mathfrak{M}}(A)$, where the $\mathfrak{M}$ is usually dropped
when the relevant structure is clear from the context.
\end{defini}

\textbf{Remark}: If $K$ is a real closed field, $X\subseteq K$, then dcl($A$) in $K$ will be the relative algebraic closure of $\mathbb{Q}(A)$ in $K$, i.e. $\text{RC}(A)$.

\begin{lemma}{\label{pe}}
Let $M$ be an $o$-minimal structure and $A\subseteq M$. Let $B$ be the set of all elements above dcl($A$) (i.e. strictly larger than all elements of dcl($A$)). Then every two elements in $B$ have the same type over $A$.
\end{lemma}
\begin{proof} Assume without loss of generality that $A=\text{dcl}(A)$. Let $a, b\in B$. It suffices to
 show that for every formula $\phi(x)$ with one free variable parameters from $A$, either both $a, b$ satisfy $\phi(x)$, or they both satisfy $\neg\phi(x)$.
 By $o$-minimality, $\phi(x)$ defines a finite union of points (in $A$) and intervals (with endpoints in ${\pm\infty} \cup A$). If the rightmost point of $S_{\phi}:=\{x\in M|M\models\phi(x)\}$ is an element of A, then, 
by assumption, both $a, b$ satisfy $\neg\phi(x)$. If not, then $S_{\phi}$ has a rightmost interval $(c, \infty)$ with $c\in A$ and, again by assumption, both $a, b$ satisfy $\phi(x)$.
\end{proof}


\begin{thm}{\label{ominmc}}
 Let $T$ be a countable, consistent and $o$-minimal theory. Then there is a transitive model $N$ of $ZF$ such that $N$ contains a model $\hat{M}\models T$ with no unbounded discrete subset.
\end{thm}
\begin{proof}
As $T$ is countable and consistent, there is a countable model $M$ of $T$ by the L\"owenheim-Skolem theorem. We will build an elementary chain $(M_{i}|i\in\omega)$ of countable models of $T$ with $M_{0}=M$.
Then we will set $\hat{M}:=\bigcup_{i\in\omega}M_{i}$. Every $M_{i}$ will be $\omega$-homogenous, hence the same will hold for $\hat{M}$ as a union of a countable elementary chain of countably homogenous models is countably homogenous.
We will have $\hat{M}\models T$ by Proposition $2.3.11$ of \cite{Ma}. As a countable union of countable sets, $\hat{M}$ will be countable.\\
The construction will be arranged in such as way that for every finite set $a\subseteq M_{i}$, the definable closure of $a$ in $M_{i+1}$ will be bounded. Thus, if $a\subseteq\hat{M}$ is finite,
the definable closure of $a$ will be bounded in $\hat{M}$.\\
 Let $X\subseteq\hat{M}$ be discrete and unbounded, and let $a\subseteq\hat{M}$ be finite. As $X$ is unbounded, let $x\in X$ be larger than
the definable closure of $a$ in $\hat{M}$. As $T$ is $o$-minimal and $x$ is not in the definable closure of $a$, there is an interval $J=(x-\varepsilon,x+\varepsilon)$ around $x$ such that
$\text{tp}(a,y)=\text{tp}(a,x)$ for all $y\in I$. By discreteness of $X$, chose $\varepsilon\in\hat{M}$ small enough such that $X\cap J=\{x\}$ and let $y\in J\setminus\{x\}$. Then by $\omega$-homogenity
there is an automorphism $\pi_{a,X}$ of $\hat{M}$ that fixes $a$ but sends $x$ to $y$. Hence $X$ has no support and as $X$ was arbitrary, no discrete unbounded subset of $\hat{M}$ has a support.
Consequently, by Lemma \ref{subsetremoval}, there is a countable transitive model $N$ of $ZF$ containing an isomorphic copy of $\hat{M}$, but no discrete unbounded subset of $X$.
As being a model of $T$ and being a discrete unbounded subset is absolute, $N\models`$There is a model of $T$ without a discrete unbounded subset', as desired.\\
Now for the construction: Starting with $M_{0}=M$, there is, (by compactness) a countable elementary extension $M_{0}^{\prime}$ of $M_{0}^{\prime}$ containing an element $x_{0}$ bigger than all elements of $M_{0}$.
We now use Proposition $4.3.6$ from \cite{Ma} to construct a countable elementary $\omega$-homogenous extension $M_{1}$ of $M_{0}^{\prime}$. If $M_{1}$ is unbounded (i.e. the definable closure of $a$ is bounded
for every finite $a\subseteq M_{1}$), we let $\hat{M}=M_1$; otherwise, we repeat the construction step. This creates a potentially infinite elementary sequence $(M_{i}|i\in\omega)$ with the desired properties.
\end{proof}

This, of course, gives both Corollary \ref{vgindep} and Corollary \ref{IndependenceOfIPs} as special cases. But we get much more:

\begin{corollary}{\label{expan}}
 It is consistent with $ZF$ that there is a model of $\mathbb{R}_{exp}$, the elementary theory of the real numbers with addition, multiplication and exponentiation without an integer part, a value group section and a residue field section.
Furthermore, it is consistent with $ZF$ that there is a model of $\mathbb{R}_{an}$, the elementary theory of the real numbers with addition, multiplication and restricted analytic functions (see \cite{DMM}) without an integer part, a value group section and
a residue field section.
\end{corollary}
\begin{proof}
It is shown in \cite{Wi} and \cite{DMM} that $\mathbb{R}_{exp}$ and $\mathbb{R}_{an}$ are $o$-minimal. Now we can apply Theorem \ref{ominmc}.
\end{proof}

\textbf{Remark}: Similarly, the same holds for real closed fields with finite Pfaffian chains, Pfaffian functions etc.

\begin{corollary}{\label{nonuniqueness}}
 Let $T$ be countable, consistent and $o$-minimal. Let $P$ be a property of discrete, unbounded subsets of models $M$ of $T$ which is expressible by a $\Delta_{0}(M)$-formula in the parameter $M$. Assume further 
that $ZFC$ proves that for each $M\models T$, there is some $X\subseteq M$ with $\phi(M,X)$. Then $ZFC$ does not prove that $X$ is unique.
\end{corollary}
\begin{proof}
In \cite{C}, the following general theorem is proved: If $ZFC$ proves $\forall{x}\exists!{y}\phi$ (where $\phi$ is $\Delta_{0}$), then already $ZF$ proves $\forall{x}\exists!{y}\phi$ (here $\exists!$ denotes `there is a unique', as usual).
 As $ZF$ does not prove existence, $ZFC$ does not prove uniqueness here.
\end{proof}

The consequence of Theorem \ref{ominmc} can be further strengthened:

\begin{corollary}{\label{ominstrong}}
Let a countable theory $T$ be consistent and $o$-minimal. Then there is a transitive model $N$ of $ZF$ containing a model $\hat{M}$ of $T$ such that, for any unbounded subset $S\subseteq\hat{M}$,
$S$ contains a final segment of $\hat{M}$ (i.e. for some $x\in\hat{M}$, $\{y\in\hat{M}:y>x\}\subseteq S$).
\end{corollary}
\begin{proof}
 Construct $\hat{M}$ as in the proof of Theorem \ref{ominmc}. Let $S$ be an unbounded subset that does not contain a final segment of $\hat{M}$; that is, the complement of $S$ in $\hat{M}$ is also unbounded in $\hat{M}$.
Assume that the finite set $a\subseteq\hat{M}$ is a support for $S$. $\text{dcl}(a)$ is bounded in $\hat{M}$, let $x\in S$, $y\in\hat{M}\setminus S$ both be greater than all elements of $\text{dcl}(a)$.
As in the proof of Theorem \ref{ominmc}, we have $\text{tp}(a,x)=\text{tp}(a,y)$, hence there is an automorphism $\pi$ of $\hat{M}$ fixing $a$ such that $\pi(x)=y$. Hence $a$ is not a support for $S$,
a contradiction. This suffices by Lemma \ref{subsetremoval}.
\end{proof}

Similarly, we can re-use the idea of Theorem \ref{rfswithoutAC} to considerably strengthen the conclusion:

\begin{corollary}{\label{rcfwithoutinfdisc}}
 There is a transitive model of $ZF$ containing a real closed field $K$ with infinite transcendence degree over $\mathbb{Q}$, but no discrete subset of infinite transcendence degree over $\mathbb{Q}$.
\end{corollary}
\begin{proof}
It suffices to observe that discreteness and infinite transcendence degree are the only properties of a residue field section used in the proof of Theorem \ref{rfswithoutAC}.
\end{proof}

\section{When supports suffice}

In this section, we consider the converse question suggested by our results above: Namely conditions under which if, in the real world, $K$ is an $RCF$ with an integer part $I$ with
a support $a$, there is an integer part of $K$ in any transitive model of $ZF$ containing $K$.

\begin{defini}
An $RCF$ $K$ is \textbf{supported} iff its transcendence degree over $\mathbb{Q}$ is finite, i.e. iff there is a finite $a\subseteq K$ such that $K=\text{RC}(a)$.
\end{defini}

It is easy to see that being supported is $\Delta_{0}$ and hence absolute between transitive models of $ZF$. We start with some easy observations:

\begin{prop}
If $K$ is supported, then every integer part of $K$ has a support.\\
Furthermore, the only automorphism $\pi:K\rightarrow K$ with $\pi(a)=a$ is the identity.
\end{prop}
\begin{proof}
Let $a\subseteq K$ be as in the definition of being supported. Then $a$ is obviously a support for every subset of $K$, including every integer part.
The second statement is also obvious.
\end{proof}

This excludes the above construction for eliminating integer parts from being applied to a supported $K$. In fact, it follows from $ZF$ that
every supported real closed fields has an integer part:

\begin{lemma}{\label{rigidIP}}
($ZF$) Let $K$ be supported. Then $K$ has an integer part.
\end{lemma}
\begin{proof}
Let $a\subseteq K$ be finite, $K=\text{RC}(a)$.
Then $K$ is in itself the Skolem hull of $a$ (with respect to formulas in the language of ordered rings).
As the formulas of the language of ordered rings are easily explicitely well-orderable in ordertype $\omega$,
so is $K$. Hence it is provable in $ZF$ that every supported $RCF$ is well-orderable in ordertype $\omega$ and hence countable. By Theorem $4.1$ of \cite{KL}, it has a residue field section (in $ZF$).
Now \cite{DKKL} shows that the construction of $IP$s for countable $RCF$s is constructible given a residue field section and a well-ordering of $K$. Hence, a supported $RCF$ $K$ has
an $IP$ in every $M\models ZF$ such that $K\in M$. By absoluteness of supportedness for $IP$s, every model of $ZF$ believes that every supported $RCF$ has an $IP$, so this is provable in $ZF$.
\end{proof}


\begin{lemma}{\label{rcdense}}
Let $K$ be a homogenous $RCF$, $I$ an integer part of $K$, $a\subseteq K$ a support for $I$. Then $K^{\prime}:=\text{RC}(a)$ is dense in $K$.
\end{lemma}
\begin{proof}
We start by observing that $K^{\prime}$ must be unbounded in $K$. If not, then there are $i\in I$ with $i>K^{\prime}$
and $j\in K\setminus I$ with $j>K^{\prime}$ such that $\text{tp}(a,i)=\text{tp}(a,j)$ and we can proceed
as in Lemma \ref{constructionofautomorphisms1} to show that $a$ is not a support for $I$, a contradiction.\\
Now assume that $K^{\prime}$ is not dense in $K$ and let $(x,y)$ be an interval of $K$ such that $K^{\prime}\cap(x,y)=\emptyset$. Withous loss
of generality, we assume that $0<x<y$. Now, as $K^{\prime}$ is unbounded in $K$, there is $d\in K^{\prime}$
such that $|dx-dy|>1$. If $z\in K^{\prime}\cap(dx,dy)$, then $zd^{-1}\in K^{\prime}\cap(x,y)$, a contradiction - thus
$K^{\prime}\cap(dx,dy)=\emptyset$. As $|dx-dy|>1$, there is $i\in I\cap(dx,dy)$. Let $r\in(dx,dy)\setminus I$
be arbitrary. Then, as in the proof of Lemma \ref{closetypes} above, $\text{tp}(a,r)=\text{tp}(a,i)$.
Hence, there is an automorphism $\pi$ of $K$ such that $\pi(a)=a$ and $\pi(i)=r$. So $a$ is not a support for $I$, a contradiction.
\end{proof}

\begin{corollary}{\label{IPtransfer}}
Under the assumptions of Lemma \ref{rcdense}, every integer part of $\text{RC}(a)$ is also an integer part of $K$.
\end{corollary}
\begin{proof}
Let $I$ be an integer part of $\text{RC}(a)$. Clearly, $I\subseteq \text{RC}(a)\subseteq K$, $I$ is a subring of $K$ with minimal element $1$
and hence discretely ordered. We need to show that each element of $K$ can be rounded down to some element of $I$.
So let $x\in K$. By density of $\text{RC}(a)$, let $x'\in \text{RC}(a)$ such that $|x-x^{\prime}|<1$, and let $i+1\in I$ such that
$i+1\leq x^{\prime}<i+2$. Then $i<x^{\prime}-1<x<x^{\prime}+1<i+3$. Consequently, we have $x\in (i,i+1]\cup(i+1,i+2]\cup(i+2,i+3)$,
so there is $j\in I$ such that $j\leq x<j+1$. As $x$ was arbitrary, $I$ is an integer part of $K$.
\end{proof}

\textbf{Remark}: Note that Lemma \ref{rcdense} and the Corollary were proved in $ZF$.\\

The following theorem shows that, in the homogenous case, supports are exactly what is needed to ensure that integer parts also exist in choice-free universes.

\begin{thm}{\label{supportssuffice}}
Let $K$ be a homogenous $RCF$ with an integer part $I$ which has a support $a$.
Then every transitive model of $ZF$ which contains $K$ also contains an integer part of $K$.
\end{thm}
\begin{proof}
Let $M\models ZF$ be transitive, $K\in M$. Work in $M$, noting that, by our remark above, everything we use is provable in $ZF$ alone and hence holds in $M$.
By Lemma \ref{rcdense}, $\text{RC}(a)$ is dense in $K$. Obviously, $\text{RC}(a)$ is supported. Hence, by Lemma \ref{rigidIP}, $\text{RC}(a)$ has an integer part $J$.
By Lemma \ref{IPtransfer}, $J$ is also an integer part of $K$. Hence $K$ has an integer part in $M$.
\end{proof}

The proof actually gives us the following:

\begin{corollary}
($ZF$) Let $K$ be an $RCF$, $a\subseteq K$ finite such that $\text{RC}(a)$ is dense in $K$. Then $K$ has an integer part.
In particular, if there is a finite $a\subseteq K$ such that $\text{RC}(a)$ contains an integer part of $K$ in any transitive model of $ZFC$, then $K$
has an integer part in every model of $ZF$ containing $K$.
\end{corollary}
\begin{proof}
Let $M\models ZF$ be transitive such that $K\in M$. Then $^{<\omega}K\in M$, and hence $a\in M$. The subfield $k$ of $K$ generated by $a$
exists and is ordered in $M$, so $\text{RC}(a)\in M$. By absoluteness of density (see Lemma \ref{absoluteness}), $\text{RC}(a)$ is dense in $K$
also in $M$. Hence, by Theorem \ref{supportssuffice}, $M$ believes that $K$ has an integer part.
\end{proof}

In the countable case, we can summarize these results as follows:

\begin{thm}{\label{choicefreeIPchar}}
Let $K$ be a countable, $\omega$-homogenous $RCF$. Then $K$ has an integer part in any transitive $M\models ZF$ with $K\in M$
iff there is an integer part of $K$ which has a support.
\end{thm}
\begin{proof}
The `if' part is Theorem \ref{supportssuffice}, the `only if' part is a direct application of Lemma \ref{subsetremoval}.
\end{proof}

\section{Remarks on Effectivity}

In \cite{H2}, W. Hodges analyzes certain field constructions in terms of `effectivity', where effectivity is taken in a weaker sense than Turing computability;
instead, he uses the primitive recursive set functions of Jensen and Karp as his underlying model of effectiveness. In this sense, our results above
allow us to argue that the construction of an integer part for a real closed field is not effective.

An account of primitive recursive set functions can be found in \cite{JK}. We recall here the definition from \cite{JK}:

\begin{defini}
A function $F$ is primitive recursive, written Prim, iff it lies in the closure of the following basic functions:
\begin{itemize}
 \item The projections $P_{n,i}$ taking $n$-tuples to their $i$th components ($1\leq i\leq n$)
 \item $F(x)=0$
 \item $F(x,y)=x\cup\{y\}$
 \item $C(x,y,u,v)=x$ if $u\in v$ and otherwise $=y$
\end{itemize}
under the following operators:
\begin{itemize}
 \item (Substitution $1$) $F(\vec{x},\vec{y})=G(\vec{x},H(\vec{x}),y)$, where $\vec{x}=(x_{1},...,x_{n})$, $\vec{y}=(y_{1},...,y_{m})$ and $m,n\in\omega$
 \item (Substitution $2$) $F(\vec{x},\vec{y})=G(H(\vec{x},\vec{y})$ (with $\vec{x},\vec{y}$ as above)
 \item (Recursion) $F(z,\vec{x})=G(\bigcup\{F(u,\vec{x}|u\in z\},z,\vec{x})$, $\vec{x}=(x_{1},...,x_{n})$, $n\in\omega$
\end{itemize}
\end{defini}

This definition is obviously analogous to the classical definition of primitive recursion when formulated for heriditarily finite sets. Restricting primitive recursive set functions to the ordinals leads to the class of 
ordinal recursive functions, which have been linked with machine models
of transfinite computations e.g. in \cite{IS} and turned out to have a considerable amount of conceptual stability that one would expect from a sensible notion of transfinite computability.
The proposal motivated in the first section \cite{H2} which we follow here is to model the intuitive concept of an effective construction via primitive recursive set functions.


\begin{lemma}{\label{prsfabs}}
Primitive recursive set functions are absolute between transitive models of $ZF$: I.e., if $\phi(x,y)$ is the defining formula for a Prim function $F$, $M_{1},M_{2}$ are transitive models of $ZF$
and $x,y\in M_{1}\cap M_{2}$, then $F(x)=y$ holds in $M_{1}$ iff it holds in $M_{2}$. Moreover, transitive models of $ZF$ are closed under Prim functions.
\end{lemma}
\begin{proof}
 See Remark $2.3$, part (4) of \cite{JK}.
\end{proof}

\begin{thm}{\label{effectiveIP}}
(1) There is no primitive recursive set function that maps each $RCF$ $K$ to an integer part of $K$.\\
(2) There is no primitive recursive set function that maps each $RCF$ $K$ to a value group section of $K$.\\
(3) There is no primitive recursive set function that maps each $RCF$ $K$ to a residue field section of $K$.
\end{thm}
\begin{proof}
(1) Take a model $M$ in which there is some $RCF$ $K$ without an $IP$. If such a function $f$ existed, it would be absolute by Lemma \ref{prsfabs}, so we would have $f(K)=f^{M}(K)\in M$.
As being an $IP$ is absolute between transitive models of $ZF$, $f(K)$ would be an $IP$ for $K$ inside $M$, which contradicts the choice of $M$.\\
(2) and (3) now follow by similar arguments, using the results of section $4.1$.
\end{proof}

Theorem \ref{effectiveIP} can be seen as complementing the results of \cite{DKKL} and \cite{KL} on the effectivity of integer parts and value group sections. They show that
the existence of integer parts and value group sections is highly nonconstructive unless the necessity of choice is eliminated by extra information. The results of \cite{DKKL} and \cite{KL} show, in
contrast, that the existence becomes highly constructive once this is done.

\section{Well-orderable real closed fields}

How much choice is actually necessary for obtaining value group sections, residue field sections and integer parts for a real closed field $K$? In this section, we give a partial answer by observing
that $ZF$ suffices to show that each well-orderable real closed field has a value group section. To this end, we generalize the arguments from
\cite{KL} dealing with the case of countable real closed fields to higher cardinalities. 

\begin{lemma}{\label{ZFWOvgs}}
($ZF$) Let $K$ be a real closed field and $\preceq$ a well-ordering of $K$. Then $K$ has a residue field section $S$.
\end{lemma}
\begin{proof}
 We adapt the proof of Theorem $3.1$ from \cite{KL}. Let $K$ be a real closed field, $K=(x_{\iota}|\iota<\kappa)$ where $\kappa$ is a cardinal and where $x_{0}=1$.
We define a sequence $(G_{\iota}|\iota<\kappa)$ of divisible abelian subgroups of $(K^{>0},\cdot)$ via transfinite recursion on $\kappa$; as the transfinite recursion principle
is provable in $ZF$, this will give the desired conclusion in $ZF$. Let $G_{0}:=\{1\}$. If $\alpha$ is a limit ordinal, then $G_{\alpha}:=\bigcup_{\iota<\alpha}G_{\iota}$.
If $\beta=\alpha+1$ is a successor ordinal, then $G_{\beta}=G_{\alpha}$ iff some element of $G_{\alpha}$ is archimedean equivalent to $x_{\alpha}$. If this is not the case,
we let $G_{\beta}$ be the subgroup of $(K,\cdot)$ generated by $G_{\alpha}\cup\{x_{\alpha}^{q}|q\in\mathbb{Q}\}$. In any case, $G_{\alpha}$ is easily seen to be a divisible abelian group.
Finally, we let $G:=\bigcup_{\iota<\kappa}G_{\iota}$. As an increasing union of divisible abelian groups, $G$ is a divisible abelian group. We show that $G$ contains a unique representative
for each archimedean equivalence class of $K$. Let $x\in K$ be arbitrary and suppose that no element of $G$ is archimedean equivalent to $x$. There is $\iota<\kappa$ such that $x=x_{\iota}$. As $x\notin G$,
we have $x=x_{\iota}\notin G_{\iota}\subseteq G$. Hence by definition $x_{\iota}\in G_{\iota+1}\subseteq G$, a contradiction. Assume now that $x,y\in G$ are archimedean equivalent. Suppose first
that $x,y$ entered $G$ simultaneously at stage $\iota$ of the construction and that $\iota$ is the minimal stage after which we have two archimedean equivalent elements in $G$.
 Hence there are $a,b\in G_{\iota}$ different from $0$ and $q_{1},q_{2}\in\mathbb{Q}$ such that $x=ax_{\iota}^{q_{1}}$ and
$y=bx_{\iota}^{q_{2}}$. If $q_{1}=q_{2}=q$, then $ax_{\iota}^{q}$ and $bx_{\iota}^{q}$ are archimedean equivalent, hence so are $a$ and $b$, which contradicts the minimality of $\iota$.
So let $q_{1}\neq q_{2}$, and assume without loss of generality that $q_{1}>q_{2}$. Then $x_{\iota}^{q_{1}-q_{2}}$ and $ba^{-1}\in G_{\iota}$ are archimedean equivalent, hence so are
$x_{\iota}$ and $(ba^{-1})^{q_{2}-q_{1}}$, so no new element would have entered $G$ at stage $\iota$, which again contradicts the minimality of $\iota$. 
Assume now that $x$ enters at stage $\iota_{1}$ and $y$ enters at stage $\iota_{2}$ where $\iota_{1}\neq \iota_{2}$, and assume without loss of generality that $\iota_{2}>\iota_{1}$.
Then there are $a\in G_{\iota_{2}}$ and $q\in\mathbb{Q}$ such that $x$ and $y=ax_{\iota_{2}}^{q}$ are archimedean equivalent; consequently, so are $x_{\iota_{2}}$ and
$(xa^{-1})^{-q}\in G_{\iota_{1}+1}\subseteq G_{\iota_{2}}$, so no new element would have entered at stage $G_{\iota_{2}}$, once more contradicting the choice of $\iota_{2}$.
\end{proof}

\textbf{Remark}: The proof of Theorem $4.1$ of \cite{KL} showing that some residue field section of a countable real closed field $K$ is $\Pi_{2}^{0}$ in $K$ can be adapted in a similar manner
to show that for an arbitrary $K$, a residue field section of $K$ is constructible in $K$ and a well-ordering of $K$, i.e. if $\preceq$ is a well-ordering of $K$, then it is provable in $ZF$ that
there is a residue field section $S$ of $K$.
The case of integer parts is considerably more involved; careful inspection of the proofs of \cite{DKKL} will probably allow one to show that an integer part of a real closed field $K$
is constructible relative to a well-ordering, a value group section and a residue field section of $K$ and hence relative to a well-ordering of $K$ alone by Theorem \ref{ZFWOvgs}
and the first part of this remark, but we will not pursue this further here.

\section{Conclusion, Open Questions and Further Work}

We have seen that there are examples of real closed fields that have neither an integer part nor a value group section or residue field sections in some transitive model of $ZF$ containing them.
This shows in particular that in the analysis of `effective' methods for constructing integer parts, value group sections and residue field sections of given real closed fields
as in \cite{DKL} or \cite{DKKL}, it is indeed necessary in general, as is done there, to cancel out the use of $AC$ by e.g. fixing a well-ordering of the real closed field. 
In contrast, we also showed that this additional assumption is unnecessary when the field in question has finite transcendence degree over $\mathbb{Q}$.\\

We do not know how strong a choice principle $\phi_{IP}$ (or the existence of value group sections or residue field sections) for real closed fields really is. In particular, we do not know whether $\phi_{IP}$ actually implies $AC$, 
(though we conjecture that it does not) or some weakening of $AC$ and how much of $AC$ is necessary for $\phi_{IP}$. Moreover, it would be interesting to see
whether e.g. $ZF+AD$ (i.e. $ZF$ with the axiom of determinacy, see e.g. \cite{J1}) implies $\phi_{IP}$.

\section{Acknowledgements}
We thank Pantelis Eleftheriou for suggesting Lemma \ref{pe}. We also thank Lorna Gregory for commenting on an earlier version of this paper and Arno Fehm for a helpful discussion of our results.



\end{document}